\documentclass{amsart}

\usepackage{amssymb,latexsym}

\usepackage{enumerate}

\newtheorem{theorem}{Theorem}[section]

\newtheorem{proposition}[theorem]{Proposition}

\theoremstyle{definition}

\numberwithin{equation}{section}

\renewcommand\leq{\leqslant}\renewcommand\geq{\geqslant}
\newcommand\Z{\ensuremath{\mathbb Z}}
\newcommand\Q{\ensuremath{\mathbb Q}}\newcommand\R{\ensuremath{\mathbb R}}

\newcommand\Qb{{\overline\Q}}

\newcommand{\ra}{{\rightarrow}}
\newcommand{\lra}{\longrightarrow}

\newcommand\End{\operatorname{End}}

\newcommand\Gal{\operatorname{Gal}}
\newcommand\GL{\operatorname{GL}}
\newcommand\Hom{\operatorname{Hom}}

\newcommand\im{\operatorname{im}}

\newcommand\Res{\operatorname{Res}}

\newcommand\acc[2]{\ensuremath{{}^{#1}\hskip-0.1ex{#2}}}

\newcommand\sign{\operatorname{sign}}

\newcommand\cB{\ensuremath{\mathcal B}}
\newcommand{\comp}{\begin{picture}(6,5)(-3,-2)\put(0,1){\circle{2}} \end{picture}}\def\circ{\comp}

\begin{document}


\baselineskip=17pt




\title[Fields of definition of building blocks with QM]{Fields of definition of building blocks with quaternionic
multiplication}

\author[X. Guitart]{Xavier Guitart}

\address{
Deptartament de Matem\`atica Aplicada II\\ Universitat Polit\`ecnica
de Catalunya\\ Carrer Colom 11\\ 08222
Terrassa}

\email{xevi.guitart@gmail.com}

\date{\today}

\begin{abstract}
This paper investigates the fields of definition up to isogeny of
the abelian varieties called \emph{building blocks}.
In~\cite{ribet-cosdef} and~\cite{pyle}  a characterization of the
fields of definition of these varieties together with their
endomorphisms is given in terms of a Galois cohomology class
canonically attached to them. However,  when the building blocks
have quaternionic multiplication, then the field of definition of
the varieties can be strictly smaller than the field of definition
of their endomorphisms. What we do is to give a characterization
of the field of definition of  the varieties in this case (also in
terms of their associated Galois cohomology class), and we also
make the computations that are needed in order to calculate in
practice these fields from our characterization.
\end{abstract}



\maketitle

\section{Introduction}

An abelian variety $B/\Qb$ is called an \emph{abelian
$\Q$-variety} if for each $\sigma \in G_\Q=\Gal(\Qb/\Q)$ there
exists an isogeny $\mu_\sigma\colon\acc\sigma B\ra B$
\emph{compatible} with the endomorphisms of $B$, i.e. such that
$\varphi\circ\mu_\sigma=\mu_\sigma\circ\acc\sigma\varphi$ for all
$\varphi\in\End_\Qb^0(B)=\End_\Qb(B)\otimes_\Z \Q$. A
\emph{building block} is an abelian  $\Q$-variety $B$ whose
endomorphism algebra $\End_\Qb^0(B)$ is a central division algebra
over a totally real number field $F$ with Schur index $t=1$ or
$t=2$ and $t[F:\Q]=\dim B$. In the case $t=2$ the quaternion
algebra is necessarily totally indefinite. The interest in the
study of the building blocks comes from the fact that they are the
 absolutely simple factors up to isogeny of the non-CM abelian varieties of
$\GL_2$-type (see~\cite{pyle}) and therefore, as a consequence of
a generalization of Shimura-Taniyama, they are the non-CM
absolutely simple factors of the modular jacobians $J_1(N)$.

In \cite{ribet-cosdef} and in \cite{pyle} Ribet and Pyle
investigated the possible fields of definition of a building block
up to isogeny; in fact, and to be more precise, their results
concern the field of definition of the variety together with its
endomorphisms. The main result in this direction is that every
building block $B/\Qb$ is isogenous over $\Qb$ to a variety $B_0$
defined over a polyquadratic number field\footnote{i.e. a
composition of quadratic extensions of $\Q$.} $K$, and with all
the endomorphisms of $B_0$ also defined over $K$ (this is
\cite[Theorem 5.1]{pyle}). Moreover, from the proof of this result
one can deduce the structure of minimal polyquadratic number
fields with this property. In particular, each of these minimal
number fields must contain a certain field $K_P$ that can be
calculated from a cohomology class $\gamma$ in $H^2(G_\Q,F^\times
)$  canonically attached to $B$.

If $B$ is a building block whose endomorphism algebra
$\End_\Qb^0(B)$ is a number field $F$ and $B$ is defined over a
number field $K$, then all the endomorphisms of $B$ are also
defined over $K$; this follows easily from  the compatibility of
the isogenies and from the commutativity of $\End_\Qb^0(B)$.
Therefore, in this case it is not a restriction to impose on a
field of definition of $B$ to be also a field of definition of its
endomorphisms. But if $B$ has quaternionic multiplication, that is
if $\End_\Qb^0(B)$ is  a quaternion algebra, then a field of
definition of $B$ is not necessarily a field of definition of
$\End_\Qb^0(B)$. In this situation, it can occur  that $B$ is
indeed isogenous to a variety $B_0$ defined over a  field $L$
smaller than the minimal ones given by Ribet and Pyle,  but of
course with $\End_L^0(B_0)$ strictly contained in $
\End_\Qb^0(B_0)$. The easiest case where this happens is in the
abelian varieties of $\GL_2$-type that are absolutely simple and
have quaternionic multiplication over $\Qb$. They are building
blocks and any field of definition of their endomorphisms is
bigger than $\Q$, but clearly $\Q$ can be taken to be a field of
definition of these varieties up to isogeny. In section~\ref{sec:
examples} we will give more involved examples of this phenomenon,
in the sense that it will not be obvious a priori if the building
block can be defined up to isogeny over a smaller field than its
endomorphisms.

 The goal of this article is to characterize the fields of definition of
quaternionic building blocks up to isogeny, and to determine under
what  conditions  it is possible to define them over a field
strictly contained in the minimal ones given by Ribet and Pyle for
the variety and the endomorphisms. The plan of the paper is as
follows. In Section \ref{sec: Building blocks and fields of
definition} we characterize the fields of definition of $B$ up to
isogeny as those $K$ such that the restriction of $\gamma$ to
$G_K$ lies in the image of a certain map $\delta\colon
\Hom(G_K,\cB^\times/F^\times)\ra H^2(G_K,F^\times) $, where $\cB$
is the quaternion algebra $\End_\Qb^0(B)$. In Section \ref{sec:
The image of delta} we compute the image of $\delta $ for the kind
of quaternion algebras $\cB$ that appear as endomorphism algebras
of building blocks. Finally, in Section \ref{sec: examples} we
apply these results and computations to determine  the field of
definition of several concrete examples of building blocks.

\section{Building blocks and fields of definition}\label{sec: Building blocks and fields of definition}
We begin this section by recalling the main tools used in the
study of  fields of definition of building blocks. The main
references for this part are~\cite{ribet-cosdef} and~\cite{pyle}
(and see also~\cite[Section 1]{quer-MC} for a similar account of
this material).

Let $K$ be a number field.  We will say that a building block $B$
is \emph{defined over $K$} if the variety $B$ (but not necessarily
all of its endomorphisms) is defined over $K$. If $B$ is isogenous
to a building block defined over $K$ we will say that $K$ is
\emph{a field of definition of $B$ up to isogeny}, or that $B$ is
\emph{defined over $K$ up to isogeny}. Note that this is a
modification of the terminology used in~\cite{pyle}, where a field
of definition of a building block was defined to be a field of
definition of the variety and of all its endomorphisms.

Our study of the fields of definition of a building block up to
isogeny will be based on the following theorem of Ribet
(cf.~\cite[Theorem 8.1]{ribet-avQ}), that characterizes such
fields.
\begin{theorem}[Ribet]\label{th:Teorema Ribet descens cos de definicio}
Let $L/K$ be a Galois extension of fields, and let $B$ be an
abelian variety defined over $L$. There exists an abelian variety
$B_0$ defined over $K$ such that $B$ and $B_{0}$ are isogenous
over $L$ if and only if there exist isomorphisms in the category
of abelian varieties up to isogeny $\{\phi_\sigma\colon \acc\sigma
B \ra B \}_{\sigma\in\Gal(L/K)}$ satisfying that
$\phi_\sigma\circ\acc\sigma\phi_\tau\circ\phi_{\sigma\tau}^{-1}=1$.
\end{theorem}

Given a building block $B$ we fix for every $\sigma\in G_\Q$ a
compatible isogeny $\mu_\sigma\colon  \acc\sigma B\ra B$.  Since
$B$ has a model defined over a number field, we can choose the
collection $\{\mu_\sigma\}$ to be locally constant. For
$\sigma,\tau\in G_\Q$ the isogeny
$c_B(\sigma,\tau)=\mu_\sigma\circ\acc\sigma\mu_\tau\circ\mu_{\sigma\tau}^{-1}$
lies in the center $F$ of $\End_\Qb^0(B)$, and the map
$(\sigma,\tau)\mapsto c_B(\sigma,\tau)$ is a continuous  2-cocycle
of $G_\Q$ with values in $F^\times $ (equipped with the trivial
$G_\Q$-action). Its cohomology class $[c_B]$ is an element of
$H^2(G_\Q,F^\times )$ that does not depend on the particular
choice of the compatible isogenies $\mu_\sigma$, and if $B\sim_\Qb
B'$ are isogenous building blocks then we can identify their
associated cohomology classes $[c_B]$ and $[c_{B'}]$. An important
property of $[c_B]$ is that it belongs to the \mbox{2-torsion}
subgroup $H^2(G_\Q,F^\times )[2]$; that is, there exists a
continuous map $\sigma\mapsto d_\sigma\colon G_\Q\ra F^\times  $
such that $c(\sigma,\tau)^2=d_\sigma d_\tau d_{\sigma\tau}^{-1}$.
The cohomology class $[c_B]$ gives all the information about the
field of definition of a building block together with its
endomorphisms up to isogeny, thanks to the following consequence
of Theorem \ref{th:Teorema Ribet descens cos de definicio}, which
is~\cite[Proposition 5.2]{pyle}.
\begin{proposition}[Ribet-Pyle] Let $B$ be a building block and
$\gamma=[c_B]$ its associated cohomology class. There exists a
variety  $B_0$ defined over a number field $K$ and with all its
endomorphisms defined over $K$ that is $\Qb$-isogenous to $B$  if
and only if $\Res_\Q^K(\gamma)=1$, where $\Res_\Q^K$ is the
restriction map $\Res_\Q^K\colon  H^2(G_\Q,F^\times )\ra
H^2(G_K,F^\times )$.
\end{proposition}

This characterization of the fields of definition of $B$ and its
endomorphisms  up to isogeny in terms of $[c_B]$ is useful because
the group $H^2(G_\Q,F^\times )[2]$ has a particularly simple
structure, that we now recall. A \emph{sign map} for $F$ is a
group homomorphism $\sign\colon F^\times \ra \{\pm 1\}$ such that
$\sign(-1)=-1$.  A sign map gives a group isomorphism $F^\times
\simeq P\times\{\pm 1\}$, where $P=F^\times /\{\pm 1\}$. From now
on we fix a $\sign$ map for $F$ by fixing an embedding of $F$ in
$\R$, and then taking the usual sign. The corresponding
isomorphism $F^\times \simeq P\times \{\pm 1\}$ gives then a
decomposition of $H^2(G_\Q,F^\times )[2]$.
\begin{proposition}\label{prop:descomposicio del H2}
Let $F$ be a totally real number field, and let $P$ be the group
$F^\times /\{\pm 1 \}$. There exists  a (non-canonical)
isomorphism of groups
\begin{equation}\label{eq: isomorfisme de la dos torsio del grup de cohomologia}
H^2(G_\Q,F^\times )[2]\simeq H^2(G_\Q,\{\pm 1\})\times
\Hom(G_\Q,P/P^2).
\end{equation}
If $\gamma=[c]\in H^2(G_\Q,F^\times )[2]$ we denote by
$\gamma_\pm\in H^2(G_\Q,\{\pm 1\})$ and $\overline{\gamma}\in
\Hom(G_\Q,P/P^2)$ its two components under the
isomorphism~\eqref{eq: isomorfisme de la dos torsio del grup de
cohomologia}. They can be computed in the following way:
\begin{enumerate}
\item The cohomology class $\gamma_\pm$ is represented by the cocycle
$(\sigma,\tau)\mapsto \sign(c(\sigma,\tau))$.
\item If
$c(\sigma,\tau)^2=d_\sigma d_\tau d_{\sigma\tau}^{-1}$ is an
expression of $c^2$ as a coboundary, the map $\overline{\gamma}$
is given by $\sigma\mapsto d_\sigma\ \mathrm{mod}\  \{\pm
1\}F^{*2}$.
\end{enumerate}
\end{proposition}
\begin{proof}
This is essentially the content of the propositions 5.3 and 5.6
in~\cite{pyle}.
\end{proof}

Let $B$ be a building block and $\gamma=[c_B]$ its associated
cohomology class. A field $K$ is a field of definition up to
isogeny of $B$ and of its endomorphisms if and only if $K$
trivializes both components $\overline \gamma$ and $\gamma_\pm$
(that is, if and only if the restriction of both components to
$G_K$ is trivial). Let $K_P$ be the fixed field of $\ker
\overline\gamma$, which is a polyquadratic extension of $\Q$. Then
$K$ trivializes $\overline \gamma$ if and only if it contains
$K_P$. Since $H^2(G_\Q,\{\pm 1 \})$ is isomorphic to the
$2$-torsion of the Brauer Group of $\Q$, we can identify
$\gamma_\pm$ with a quaternion algebra over $\Q$, and $K$
trivializes $\gamma_\pm$ if and only if it is a splitting field of
the quaternion algebra represented by $\gamma_\pm$. If $K_P$
already trivializes $\gamma_\pm$, then $K_P$ is the minimum field
of definition of $B$ and of its endomorphisms up to isogeny.
Otherwise, there is no such a minimum field: all the fields of
definition of $B$ and of its endomorphisms up to isogeny must
contain $K_P$ and are splitting fields of $\gamma_\pm$. For
instance, for each maximal subfield $K_\pm$ of the quaternion
algebra given by $\gamma_\pm$, the field $K_\pm K_P$ is a minimal
polyquadratic number field with the property of being a field of
definition of $B$ and of its endomorphisms up to isogeny.

We can also use the cohomology class $[c_B]$ in order to study the
fields of definition of $B$ up to isogeny, in a similar way as
it is done for the fields of definition of $B$ and of its
endomorphisms. From now on we assume that  $\cB=\End_\Qb^0(B)$ is
a quaternion algebra. Before stating our cohomological version of
Theorem \ref{th:Teorema Ribet descens cos de definicio} for
building blocks, we recall that  the exact sequence of trivial
$G_K$-modules
$$1\buildrel{}\over\lra F^\times \buildrel{}\over\lra
\cB^\times\buildrel{}\over\lra
\cB^\times/F^\times\buildrel{}\over\lra 1$$
 gives rise to the cohomology exact sequence of pointed sets (cf.
\cite[p. 125]{Se})
\begin{equation*} \cdots\lra H^1(G_K,F^\times) \buildrel{}\over\lra
H^1(G_K,\cB^\times) \buildrel{}\over\lra H^1(G_K,\cB^\times
/F^\times) \buildrel{\delta}\over\lra
H^2(G_K,F^\times).\end{equation*} Since we consider the trivial
$G_K$-action, we can identify $H^1(G_K,\cB^\times /F^\times)$ with
$\Hom(G_K,\cB^\times /F^\times)$ up to conjugation.  The explicit description of the
connecting map $\delta$ is given in terms of cocycles by
\begin{equation}\label{eq: descripcio explicita
de delta particularitzada a building blocks}
\begin{array}{clcl}
 \delta\colon  & \Hom(G_K,\cB^\times/F^\times)&\lra & H^2(G_K,F^\times)\\
  & [\sigma\mapsto \psi_\sigma F^\times]&\longmapsto & [(\sigma,\tau)\mapsto \psi_\sigma\circ
  \psi_\tau\circ
  \psi_{\sigma\tau}^{-1}].
\end{array}
\end{equation}

\begin{proposition}\label{prop: resultat principal}
Let $B$ be a building block and $\gamma=[c_B]\in H^2(G_\Q,F^\times
)$ its associated \mbox{cohomology} class. There exists a variety
$B_0$  defined over a number field $K$ that is $\Qb$-isogenous to
$B$ if and only if there exists a continuous  morphism
$\psi:G_K\ra \cB^\times /F^\times $ such that
$\Res_\Q^K(\gamma)=\delta(\psi)$.
\end{proposition}
\begin{proof}
By Theorem \ref{th:Teorema Ribet descens cos de definicio} the
existence of a variety $B_0$ defined over $K$ and isogenous to $B$
is equivalent to the existence of isomorphisms of abelian
varieties up to isogeny $\phi_\sigma:\acc\sigma B\ra B$ such that
\begin{equation}\label{eq: cond ribet}
\phi_\sigma\circ\acc\sigma\phi_\tau\circ\phi_{\sigma\tau}^{-1}=1,\end{equation}
for all $\sigma,\tau\in G_K$. If $\mu_\sigma:\acc\sigma B\ra B$ is
a compatible isogeny, then $\phi_\sigma$ is equal to
$\psi_\sigma\circ\mu_\sigma$ for some $\psi_\sigma$ belonging to
$\cB^\times$. Using the compatibility of the $\mu_\sigma$'s we
have that \eqref{eq: cond ribet} is then equivalent to
$$\mu_\sigma\circ\acc\sigma\mu_\tau\circ\mu_{\sigma\tau}^{-1}\circ\psi_\sigma\circ\psi_\tau\circ\psi_{\sigma\tau}^{-1}=1
,$$ for all $\sigma,\tau\in G_K$.
 Since $c_B(\sigma,\tau)=\mu_\sigma\circ\acc\sigma\mu_\tau\circ\mu_{\sigma\tau}$
belongs to $F^\times$, we see that the map $\sigma\mapsto
\psi_\sigma F^\times$ is a morphism $\psi:G_K\ra
\cB^\times/F^\times$, and that
$\Res_\Q^K([c_B])\cdot\delta(\psi)=1$. From this the result
follows, because $[c_B]$ is a $2$-torsion element.
\end{proof}

Now suppose that $K$ is a minimal polyquadratic field of
definition of $B$ and of all its endomorphisms. As we have seen,
it might exist a variety $B_0$ defined over a subfield $L$ of $K$
that is isogenous to $B$, but in this case with
$\End_L^0(B_0)\varsubsetneq \End_\Qb^0(B_0)$. An interesting case
of this situation is when the endomorphisms of $B_0$ are defined
over $K$, but then the field $L$ cannot be much smaller than $K$,
as we can see in the following
\begin{proposition}\label{prop: descens del cos de definicio mantenint el cos de
definicio dels endomorfismes} Let $B$ be a building block such
that $B$ and its endomorphisms are defined  over a minimal
polyquadratic field $K$. Let $L\varsubsetneq K$ and let $B_0$ be an abelian variety over $L$. The abelian variety $B_0$
is $K$-isogenous to $B$ and has all of its endomorphisms defined over $K$ if and only if there exists
a continuous homomorphism $\psi:G_L\ra \cB^\times/F^\times$ such
that $\Res_\Q^L(\gamma)=\delta(\psi)$ and $G_K\subseteq
\ker(\psi)$. In particular $\Gal(K/L)\simeq C_2$ or
$\Gal(K/L)\simeq C_2\times C_2$.
\end{proposition}
\begin{proof} Let $\kappa:B\ra B_0$ be an isogeny defined over $K$, where
$B_0$ is defined over $L$ and $\End_\Qb^0(B_0)=\End^0_K(B_0)$. For
$\sigma\in G_L $ let
$\nu_\sigma=\kappa^{-1}\circ\acc\sigma\kappa$, and let
$\psi_\sigma=\nu_\sigma\circ\mu_\sigma^{-1}$ where $\mu_\sigma$ is
a compatible isogeny for $B$. Since
$\nu_\sigma\circ\acc\sigma\nu_\tau\circ\nu_{\sigma\tau}^{-1}=1$
for all $\sigma,\tau \in G_L$, we see that
$\Res_\Q^L(\gamma)=\delta(\psi)$. Moreover, for $\sigma\in G_K$
the isogeny $\mu_\sigma$ lies in $F^\times$ and $\nu_\sigma=1$, so
$\psi_\sigma$ belongs to $F^\times$.

 For the other implication, for
$\sigma\in G_L$ let $\nu_\sigma=\psi_\sigma\circ\mu_\sigma$, with
$\mu_\sigma$ a compatible isogeny. Under the conditions of the
proposition, there exists a variety $B_0$ defined over $L$ and an
isogeny $\kappa:B\ra B_0$ such that
$\nu_\sigma=\kappa^{-1}\circ\acc\sigma\kappa$. Then any
endomorphism of $B_0$ is of the form
$\kappa\circ\varphi\circ\kappa^{-1}$ for some $\varphi\in
\End_\Qb^0(B)$. Then for $\sigma\in G_K$ we have that
\begin{eqnarray*}
\acc\sigma(\kappa\circ\varphi\circ\kappa^{-1})&=&\acc\sigma\kappa\circ\acc\sigma\varphi\circ\acc\sigma\kappa^{-1}
=\kappa\circ\psi_\sigma\circ\mu_\sigma\circ\acc\sigma
\varphi\circ\mu_\sigma^{-1}\circ\psi_\sigma^{-1}\circ\kappa^{-1}\\&=&\kappa\circ\psi_\sigma\circ
\varphi\circ\psi_\sigma^{-1}\circ\kappa^{-1}=\kappa\circ\varphi\circ\kappa^{-1}.
\end{eqnarray*}
Finally, the last statement follows because $\Gal(K/L)$ must be
isomorphic to a subgroup of $\cB^\times/F^\times$, and all abelian
groups of exponent $2$  contained in $\cB^\times/F^\times$ are
isomorphic to
 either $C_2$ or $C_2\times C_2$ (see Proposition \ref{prop: subgrups finits de B*/F*} for a classification of all
finite subgroups of $\cB^\times/F^\times$).
\end{proof}

\section{The image of $\delta$}\label{sec: The image of delta}
This section is devoted to compute all the elements in
$H^2(G_K,F^\times )[2]$ that are of the form $\delta(\psi)$ for
some continuous morphism $\psi\colon G_K\ra \cB^\times /F^\times
$, and to determine their components $\delta(\psi)_\pm$ and
$\overline{\delta(\psi)}$ under the isomorphism $H^2(G_K,F^\times
)[2]\simeq H^2(G_K,\{\pm 1\})\times \Hom (G_K,P/P^2)$ (this
isomorphism is just the restriction of~\eqref{eq: isomorfisme de
la dos torsio del grup de cohomologia} to $G_K$). The image of a
continuous morphism $\psi\colon G_K\ra \cB^\times /F^\times $ is a
finite subgroup of $\cB^\times /F^\times $. In~\cite[Section
2]{C-F} these subgroups are studied and,  in particular, we have
the following result.
\begin{proposition}[Chinburg-Friedman]\label{prop: subgrups finits de B*/F*} Let $\cB$ be a totally indefinite division
quaternion
algebra over a field $F$. The finite subgroups of $\cB^\times
/F^\times $ are cyclic or dihedral. There always exist subgroups
of $\cB^\times /F^\times $ isomorphic to $C_2$ and $C_2\times
C_2$. For $n>2$, if $\zeta_n$ is a primitive $n$-th root of unity
in $\overline F$, $\cB^\times /F^\times $ contains a subgroup
isomorphic to the cyclic group $C_n$ of order $n$ if and only if $\zeta_n+\zeta_n^{-1}$
belongs
to  $ F$ and $F(\zeta_n)$ is isomorphic to a maximal subfield of
$\cB$. In this case, $\cB^\times /F^\times $ always contains a
subgroup isomorphic to the dihedral group $D_{2n}$ of order $2n$.
\end{proposition}
In order to compute the cohomology classes $\delta(\psi)$ we will
consider four separate cases, depending on whether $\im \psi$ is
isomorphic to $C_2$, $C_2\times C_2$, $C_n$ or $D_{2n}$ for $n>2$.
The following notation may be useful: if $G$ is a group, we denote
by $\Delta_G$ the elements $\gamma\in H^2(G_K,F^\times )[2]$ that
are of the form $\gamma=\delta(\psi)$ for some morphism $\psi$
with $\im \psi\simeq G$.

As usual we will identify the elements in $H^2(G_K,\{\pm 1\})$
with quaternion algebras over $K$, and we will use the notation
$(a,b)_K$ for the quaternion algebra generated over $K$ by $i,j$
with $i^2=a$, $j^2=b$ and $ij+ji=0$. As for the elements in
$\Hom(G_K,P/P^2)$ we will use the notation $(t,d)_P$  with $t\in
K$ and $d\in F^\times $, to denote (the inflation of)  the
morphism that sends the non-trivial automorphism of
$\Gal(K(\sqrt{t})/K)$ to the class of $d$ in $P/P^2$. Every
element in $\Hom(G_K,P/P^2)$ is the product of morphisms of this
kind, and therefore it can be expressed in the form
$(t_1,d_1)_P\cdot(t_2,d_2)_P\cdots (t_n,d_n)_P$ for some $t_i\in
K$, $d_i\in F^\times $. We remark that, although they are
convenient for their compactness, these expressions for the
elements of $\Hom(G_K,P/P^2)$ are not unique.

\begin{proposition}\label{prop: C2}
An element  $\gamma\in H^2(G_K,F^\times)[2]$ belongs to
$\Delta_{C_2}$ if and only if
\begin{itemize}
\item $\overline \gamma=(t,b)_P$, for some $t\in K\setminus K^2$ and
$b\in F^\times$  such that $F(\sqrt b)$ is isomorphic to a maximal
subfield of $\cB$.
\item $\gamma_\pm=(t,\sign(b))_K$.
\end{itemize}
\end{proposition}
\begin{proof}
Let $\psi$ be a morphism whose image is isomorphic to $C_2$. Then
the fixed field of $\ker \psi$ is $K(\sqrt{t})$ for some $t\in
K\setminus K^2$, and $\psi$ is the inflation of a morphism (that
we also call $\psi$) from $\Gal(K(\sqrt{t})/K)$, which is
determined by the image of a generator $\sigma$ of the Galois
group. If $\psi(\sigma)= \overline{y}$ (here $\overline{y}$ means
the class of $y$ in $\cB^\times /F^\times $), then $y^2=b\in
F^\times $ and $y\notin F^\times $. That is,  $F(\sqrt{b})$ is
isomorphic to a maximal subfield of $\cB$. From the explicit
description  of $\delta$ given in~\eqref{eq: descripcio explicita
de delta particularitzada a building blocks}, a straightforward
computation shows that a cocycle $c$ representing $\delta(\psi)$
is given by
\begin{equation*}
c(1,1)=c(1,\sigma)=c(\sigma,1)=1,\ \ c(\sigma,\sigma)=b.
\end{equation*}
By taking the sign of this cocycle we obtain a representant for
$\delta(\psi)_\pm$, and it corresponds to the quaternion algebra
$(t,\sign (b))_K$. The cocycle $c^2$ is the coboundary of the map
$1\mapsto  1$, $\sigma\mapsto b$, and by Proposition
~\ref{prop:descomposicio del H2} the component
$\overline{\delta(\psi)}$ is $(t,b)_P$.

Now, for $t\in K\setminus K^2$ and $b\in F^\times $ such that
$F(\sqrt{b})$ is isomorphic to a maximal subfield of $\cB$, take
$y\in \cB$ with $y^2=b$. Then the morphism $\psi\colon
\Gal(K(\sqrt{t})/K)\ra \cB^\times /F^\times $ that sends a
generator $\sigma$ to $\overline{y}$ has image isomorphic to
$C_2$, and by the previous argument the components of
$\delta(\psi)$ are $\delta(\psi)_\pm=(t,\sign(b))_K$ and
$\overline{\delta(\psi)}=(t,b)_P$.
\end{proof}
\begin{proposition}\label{prop: C2xC2}
An element $\gamma\in H^2(G_K,F^\times)[2]$ lies in
$\Delta_{C_2\times C_2}$ if and only if
\begin{itemize}
\item $\overline \gamma=(s,a)_P\cdot (t,b)_P$ for some  $s,t\in K\setminus
K^2$ and $a,b\in F$ such that $a$ is positive and  $\cB\simeq
(a,b)_F$.

\item $\gamma_\pm=(\sign(b) s, t)_K$.
\end{itemize}
\end{proposition}
\begin{proof}
If $\psi$ is a morphism with image isomorphic to $C_2\times C_2$,
it factorizes through a finite Galois extension $M/K$ with
$\Gal(M/K)\simeq C_2\times C_2$. We write $M$ as
$M=K(\sqrt{s},\sqrt{t})$, and let $\sigma,\tau$ be the generators
of the Galois group such that
$M^{\langle\sigma\rangle}=K(\sqrt{t})$ and
$M^{\langle\tau\rangle}=K(\sqrt{s})$. If $\overline
x=\psi(\sigma)$ and $\overline y=\psi(\tau)$, we know that
$x^2=a\in F^\times $, $y^2=b\in F^\times $ and $xy=\varepsilon y
x$ for some $\varepsilon \in F^\times $. In fact, multiplying this
expression on the left by $x$  we see that necessarily
$\varepsilon = -1$, and hence $\cB\simeq (a,b)_F$.

Let $\gamma_{s,a}$ be the cocycle in $Z^2(\Gal(M/K),F^\times )$
defined as the inflation of the cocycle
\begin{equation*}
\gamma_{s,a}(1,1)=\gamma_{s,a}(\sigma,1)=\gamma_{s,a}(1,\sigma)=1,\
\ \gamma_{s,a}(\sigma,\sigma)=a,
\end{equation*}
and in a similar way we define the cocycle $\gamma_{t,b}$ by means
of
\begin{equation*}
\gamma_{t,b}(1,1)=\gamma_{t,b}(\tau,1)=\gamma_{t,b}(1,\tau)=1,\ \
\gamma_{t,b}(\tau,\tau)=b.
\end{equation*}
Let $\chi_s$ and $\chi_t$ be the elements in
$\Hom(\Gal(M/K),\Z/2\Z)$ defined by
${{^\rho\sqrt{s}}}/{\sqrt{s}}=(-1)^{\chi_s(\rho)}$ and
${{^\rho\sqrt{t}}}/{\sqrt{t}}=(-1)^{\chi_t(\rho)}$, and let
$\gamma_{s,t}$ be the 2-cocycle defined by
$\gamma_{s,t}(\rho,\mu)=(-1)^{\chi_s(\mu)\chi_t(\rho)}$. Then, a
direct computation shows that a cocycle representing
$\delta(\psi)$ is the product of these three 2-cocycles:
$c=\gamma_{s,t}\cdot \gamma_{s,a}\cdot\gamma_{t,b}$. The cocycle
$\gamma_{s,t}$ represents the quaternion algebra $(s,t)_K$, and
then we have that $\delta(\psi)_\pm=(s,t)_K\cdot (s,\sign
(a))_K\cdot (t,\sign (b))_K$. Since $\cB$ is totally indefinite,
we can suppose that $a$ is  positive, and then
$\delta(\psi)_\pm=(\sign(b)\, s,t)_K$. Arguing as in the proof
of~\ref{prop: C2}, the component $\overline{\delta(\psi)}$ is
easily seen to be $(s,a)_P\cdot (t,b)_P$.

Finally, suppose that $\cB\simeq (a,b)_F$ where the element $a$ is
positive. Let $s,t$ be in $K\setminus K^2$, and let $x,y\in \cB$
be such that $x^2=a$, $y^2=b$ and $xy=-yx$. With the same
notations as before for $\Gal(K(\sqrt{s},\sqrt{t})/K)$, the map
$\psi$ that sends $\sigma $ to $\overline x$ and $\tau$ to
$\overline y$ satisfies that $\delta(\psi)_\pm=(\sign(b)\,s,t)_K$
and $\overline{\delta(\psi)}=(s,a)_P\cdot(t,b)_P$.

\end{proof}
\begin{proposition}
Suppose that $\cB^\times /F^\times $ contains a subgroup
isomorphic to $C_n$ for some $n>2$, and let  $\zeta_n$ be a
primitive $n$-th root of unity in $\overline F$ and
$\alpha=2+\zeta_n+\zeta_n^{-1}$. An element $\gamma\in
H^2(G_K,F^\times )$ lies in $\Delta_{C_n}$ if and only if there
exists a cyclic extension $M/K$, with
$\Gal(M/K)=\langle\sigma\rangle$ such that
\begin{itemize}
\item $\overline{\gamma}=(t,\alpha)$, where
$M(\sqrt{t})=M^{\langle\sigma^2\rangle}$.
\item $\gamma_\pm$ is represented by the cocycle
\begin{eqnarray}\label{eq:cocycle cyclic algebra -1}
c_\pm(\sigma^i,\sigma^j)=
\begin{cases}
1 & \text{if $i+j<n$,}
\\
-1 & \text{if $i+j\geq n$,}
\end{cases}
\end{eqnarray}
\end{itemize}
We note that if $n$ is odd then $\Delta_{C_n}=\{1\}$.
\end{proposition}
\begin{proof}
Let $\psi$ be a  morphism with image isomorphic to $C_n$. Then the
fixed field for $\ker\psi$ is a cyclic extension $M/K$ with
$\Gal(M/K)=\langle\sigma\rangle$. The element $x\in \cB^\times $
such that $\psi(\sigma)=\overline x$ has the property that $a=x^n$
lies in $F^\times $. Since $\psi(\sigma^i)=\overline{x^i}$, a
straightforward computation shows that $\delta(\psi)$ is given by
\begin{eqnarray}\label{eq:cocicle Cn}
c(\sigma^i,\sigma^j)=
\begin{cases}
1 & \text{if $i+j<n$,}
\\
a & \text{if $i+j\geq n$.}
\end{cases}
\end{eqnarray}
By~\cite[Lemma 2.1]{C-F} we can suppose that $x=1+\zeta$ with
$\zeta\in \cB^\times $  an element of order $n$. We  identify
$\zeta$ with $\zeta_n$ and then by Proposition ~\ref{prop:
subgrups finits de B*/F*} we see that $\zeta+\zeta^{-1}\in
F^\times $. From $(1+\zeta)^2\zeta^{-1}=2+\zeta+\zeta^{-1}$  we
see that $(1+\zeta)^{2n}=(2+\zeta+\zeta^{-1})^n$, and if we define
$\alpha=(2+\zeta+\zeta^{-1})\in F^\times $ we have that
$a^2=x^{2n}=(1+\zeta)^{2n}=\alpha^n$. Therefore, the cocycle $c^2$
is the coboundary of the map $\sigma^{i}\mapsto \alpha^i$, $0\leq
i< n$, and by Proposition \ref{prop:descomposicio del H2} the
component $\overline{\delta(\psi)}$ is the map that sends $\sigma$
to the class of $\alpha$ in $P/P^2$. Clearly $\sigma^2$ is in the
kernel of this map, and since
$\langle\sigma\rangle=\langle\sigma^2\rangle$ if $n$ is odd, then
$\overline{\delta(\psi)}$ is trivial in this case, while if $n$ is
even and $K(\sqrt{t})$ is the fixed field of $M$ by
$\langle\sigma^2\rangle$, then
$\overline{\delta(\psi)}=(t,\alpha)_P$.

A cocycle representing $\delta(\psi)_\pm$ is the sign
of~\eqref{eq:cocicle Cn}. If $n$ is odd, the cohomology class of
this cocycle is always trivial (it is the coboundary of the map
$\sigma^i\mapsto (\sign a)^i$ for $0\leq i<n$). If $n$ is even
then $a$ is negative because
$$a=x^n=(1+\zeta)^n=(2+\zeta+\zeta^{-1})^{n/2}\zeta^{n/2}=-(2+\zeta+\zeta^{-1})^{n/2},$$
and $2+\zeta+\zeta^{-1}$ is  positive due to the identification of
$\zeta$ with $\zeta_n$. This gives that $\delta(\psi)_\pm$ is
given by~\eqref{eq:cocycle cyclic algebra -1}.

Finally, if $t$, $M$, $\sigma$ and $\alpha$ are as in the
statement of the proposition, the map $\psi$ sending $\sigma$ to
$\overline{(1+\zeta)}$  with $\zeta\in B^\times $  an element of
order $n$  gives a $\delta(\psi)$ with the predicted components.
\end{proof}

\begin{proposition}\label{prop: diedral}
Suppose that $\cB^\times /F^\times $ contains a subgroup
isomorphic to $D_{2n}$ for some $n>2$. Let $\zeta_n$ be a
primitive $n$-th root of unity in $\overline F$,
$\alpha=2+\zeta_n+\zeta_n^{-1}$ and
$d=(\zeta_n+\zeta_n^{-1})^2-4$. A cohomology class $\gamma\in
H^2(G_K,F^\times )$ lies in $\Delta_{D_{2n}}$ if and only if there
exists a dihedral extension $M/K$, with
$\Gal(M/K)=\langle\sigma,\tau\ | \
 \sigma^n=1,\tau^2=1,\sigma\tau=\tau\sigma^{-1}\rangle$ such that
\begin{itemize}
\item $\overline \gamma=(s,\alpha)_P\cdot(t,b)_P$, where
$L(\sqrt{s})=M^{\langle\sigma^2,\tau\rangle}$,
$L(\sqrt{t})=M^{\langle\sigma\rangle}$ and $b\in F^\times $
satisfies that $\cB\simeq (d,b)_F$.
\item $\gamma_\pm$ is
 given by the cocycle
\begin{eqnarray}\label{eq:delta psi +- en el cas D2n}
c_\pm(\sigma^i\tau,\sigma^{i'}\tau^{j'})=
\begin{cases}
1 & \text{if $\ \ i-i'\geq 0$}
\\
-1 & \text{if $\ \ i-i'< 0$,}
\end{cases}
\end{eqnarray}
\begin{eqnarray*}
c_\pm(\sigma^i,\sigma^{i'}\tau^{j'})=
\begin{cases}
1 & \text{if $\ \ i+i'< n$}
\\
 -1 & \text{if $\ \ i+i'\geq n$,}
\end{cases}
\end{eqnarray*}
\end{itemize} We note that if $n$ is
odd, then $\overline\gamma=(t,b)_P$ and $\gamma_\pm=1$.
\end{proposition}
\begin{proof}
Let $\psi$ be a morphism with image isomorphic to $D_{2n}$. It
factorizes through a dihedral extension $M$ with
$\Gal(M/K)=\langle\sigma,\tau\rangle$ and the relations between
the generators as in the proposition. If we call
$\overline{x}=\psi(\sigma)$, $\overline{y}=\psi(\tau)$, we know
that $x^n=a\in F^\times $, $y^2=b\in F^\times $ and there exist
some $\varepsilon\in F^\times $ such that $xy=\varepsilon
yx^{-1}$. Multiplying in the left by $x^{n-1}$ we find that
$x^ny=\varepsilon^n y x^{-n}$ and hence $\varepsilon^n=a^2$. Now
we show that, in fact, $\varepsilon$ can be identified with
$\alpha$. Indeed, $x=1+\zeta$ with $\zeta\in \cB^\times $ of order
$n$ that we identify with $\zeta_n$, and so
$x^{-1}=(1+\zeta^{-1})(2+\zeta+\zeta^{-1})^{-1}$. Since $F(\zeta)$
is a maximal subfield of $\cB$ different from $F(y)$, the
conjugation by $y$ is a non-trivial automorphism of $F(\zeta)/F$.
The only such automorphism is complex conjugation, which sends
$\zeta$ to $\zeta^{-1}$, and therefore $y^{-1}\zeta y=\zeta^{-1}$.
This implies that $(1+\zeta)y=y(1+\zeta^{-1})$, and this is
$xy=(2+\zeta+\zeta^{-1})yx^{-1}$ which proves that $\varepsilon
=(2+\zeta+\zeta^{-1})$, which is identified  with  $\alpha$.

 In
order to give a compact expression for $\delta(\psi)$ we first
define a cocycle $\gamma_b$:
\begin{equation*}
\gamma_b(\sigma^i\tau^j,\sigma^{i'}\tau^{j'})=
\begin{cases} 1 & \text{if $j+j'<2$,}
\\
b &\text{if $j+j'=2$.}
\end{cases}
\end{equation*}
and a cocycle $e$:
\begin{eqnarray*}
e(\sigma^i\tau,\sigma^{i'}\tau^{j'})=
\begin{cases}
\alpha^{i'} & \text{if $i-i'\geq 0$}
\\
\alpha^{i'} a^{-1} & \text{if $i-i'< 0$,}
\end{cases}\ \ \ \ \
e(\sigma^i,\sigma^{i'}\tau^{j'})=
\begin{cases}
1 & \text{if $i+i'< n$}
\\
 a & \text{if $i+i'\geq n$.}
\end{cases}
\end{eqnarray*}
To compute a cocycle that represents $\delta(\psi)$, we take the
lift  $\tilde\psi$ form $\cB^\times /F^\times $ to $\cB$ given by
$\tilde{\psi}(\sigma^i\tau^j)=x^iy^j$ for $0\leq i < n$, $0\leq j
< 2 $. Then we have that
\begin{eqnarray*}
(\delta(\psi))(\sigma^i\tau,\sigma^{i'}\tau^{j'})&=&\tilde{\psi}(\sigma^i\tau)\tilde{\psi}(\sigma^{i'}\tau^{j'})\tilde{
\psi}(\sigma^i\tau
\sigma^{i'}\tau^{j'})^{-1}\\&=&\tilde{\psi}(\sigma^i\tau)\tilde{\psi}(\sigma^{i'}\tau^{j'})\tilde{\psi}(\sigma^{i-i'}
\tau^{1+j'})^{-1}\\
&=&\begin{cases} x^iyx^{i'}y^{j'}(x^{i-i'}y^{(1+j')\, \text{mod}\,
2})^{-1} & \text{if $i-i'\geq 0$}
\\
x^iyx^{i'}y^{j'}(x^{n+(i-i')}y^{(1+j')\, \text{mod}\, 2})^{-1} &
\text{if $i-i'< 0$}
\end{cases}\\
&=&\begin{cases} \alpha^{i'}x^{i-i'}y^{1+j'}y^{-(1+j')\,
\text{mod}\, 2}\ x^{-(i-i')} & \text{if $i-i'\geq 0$}
\\
 \alpha^{i'}x^{i-i'}y^{1+j'}y^{-(1+j')\,
\text{mod}\, 2}\ x^{-(i-i')}x^{-n} & \text{if $i-i'<0$}
\end{cases}\\
&=&
\begin{cases}
\gamma_b(\sigma^i\tau,\sigma^{i'}\tau^{j'})\alpha^{i'} & \text{if
$i-i'\geq 0$}
\\
 \gamma_b(\sigma^i\tau,\sigma^{i'}\tau^{j'})\alpha^{i'} a^{-1} & \text{if $i-i'< 0$}.
\end{cases}
\end{eqnarray*}
\begin{eqnarray*}
(\delta(\psi))(\sigma^i,\sigma^{i'}\tau^{j'})&=&\tilde{\psi}(\sigma^i)\tilde{\psi}(\sigma^{i'}\tau^{j'})\tilde{\psi}
(\sigma^i
\sigma^{i'}\tau^{j'})^{-1}\\&=&\tilde{\psi}(\sigma^i)\tilde{\psi}(\sigma^{i'}\tau^{j'})\tilde{\psi}(\sigma^{i+i'}\tau^{
j'})^{-1}\\
&=&\begin{cases} x^ix^{i'}y^{j'}(x^{i+i'}y^{j'})^{-1} & \text{if
$i+i'< n$}
\\
x^ix^{i'}y^{j'}(x^{(i+i')-n}y^{j'})^{-1} & \text{if $i+i'\geq n$}
\end{cases}\\
&=&\begin{cases} x^{i+i'}y^{j'}y^{-j'}x^{-(i+i')} & \text{if
$i+i'< n$}
\\
x^{i+i'}y^{j'}y^{-j'}x^{-(i+i')}x^n & \text{if $i+i'\geq n$}
\end{cases}\\
&=&
\begin{cases}
\gamma_b(\sigma^i,\sigma^{i'}\tau^{j'}) & \text{if $i+i'< n$}
\\
 \gamma_b(\sigma^i,\sigma^{i'}\tau^j)\cdot  a & \text{if $i+i'\geq n$}.
\end{cases}
\end{eqnarray*}
From these expressions we see that $\delta(\psi)$ is represented
by the cocycle $\gamma_b\cdot e$. Clearly $\gamma_b$ is 2-torsion
since $\gamma_b^2$ is the coboundary of the map
$d_\gamma(\sigma^i)= 1$, $d_\gamma(\sigma^i\tau)= b$. The cocycle
$e$ is 2-torsion as well, and a coboundary for $e^2$ is given by
the map $d_e(\sigma^i\tau^j)=\alpha^i$.  If we view $d_\gamma$ and
$d_e$ as taking values in $P/P^2$, then by Proposition
\ref{prop:descomposicio del H2} we have that
$\overline{\delta(\psi)}$ is the map $d_e\cdot d_\gamma$. Note
that $\langle\sigma^2,\tau\rangle\subseteq \ker d_e$. If $n$ is
odd, then $\langle\sigma^2,\tau\rangle=\Gal(M/K)$ and the only
contribution to $\overline{\delta(\psi)}$ comes from $d_\gamma$,
and it is the map $(t,b)_P$. If $n$ is even, then the contribution
from $d_e$ is $(s,\alpha)$, and in this case
$\overline{\delta(\psi)}=(s,\alpha)_P\cdot (t,b)_P$.

The component $\delta(\psi)_\pm$ comes from taking the sign in the
cocycle $\gamma_b\cdot e$. The element $b$ is  positive, since
by~\cite[Lemma 2.3]{C-F} we have that $\cB\simeq (d,b)_F$, and $d$
is negative. To determine the sign of $a$, note that from
$\alpha^n=a^2$, we have that if $n$ is even then $\alpha^{n/2}=\pm
a$. The case $\alpha^{n/2}=a$ is not possible since otherwise
$F(x^{n/2},y)$ would be a subfield of $\cB$ of dimension 4 over
$F$. Then $\alpha^{n/2}=-a$ and the fact that $\alpha$ is totally
positive forces $a$ to be  negative. This gives that
$\delta(\psi)_\pm$ is represented by the cocycle~\eqref{eq:delta
psi +- en el cas D2n}. If $n$ is odd then $c_\pm$ is the
coboundary of the map $\sigma^i\tau^j\mapsto (-1)^i$.

As usual, given an extension $M/K$, elements $b\in F^\times $,
$s,t\in K^\times $ and $c_\pm\in Z^2(\Gal(M/K),\{\pm 1\})$ with
the properties described  in  the proposition, one can construct
easily a map $\psi$ with the prescribed $\delta(\psi)$ just
defining $\psi(\sigma)=\overline x$ and $\psi(\tau)=\overline y$,
where $\overline x,\overline y$ generate a subgroup of $\cB^\times
$ isomorphic to $D_{2n}$ and $y^2=b$.
\end{proof}

\section{Examples}\label{sec: examples}
In this section we  illustrate with some examples the use of the
techniques developed so far in studying  the field of definition
of building blocks up to isogeny. We will use the information
provided by the building block table  of~\cite[Section 5.1 of the
Appendix]{quer-MC}. These data  can also be obtained directly by
means of the {\tt Magma} functions implemented by Jordi Quer,
which are based on the packages of William Stein for modular
abelian varieties.

\subsection*{Example.} Let $B$ be the only building block of
dimension 2 with quaternionic multiplication that is associated to
a newform $f$ of level $N=243$ and trivial Nebentypus, and let
$\gamma=[c_B]$ be its
 cohomology class. The components of $\gamma$ are
$\gamma_\pm=1$ and $\overline\gamma=(-3,6)_P$, and
$K_P=\Q(\sqrt{-3})$ is a minimum field of definition of $B$ and of
its endomorphisms up to isogeny. The dimension of $B$ is 2, as it
is the dimension of $A_f$; therefore, we know a priori that $\Q$
is a field of definition of $B$ up to isogeny. Let us see now how
this can also be deduced using our results. The endomorphism
algebra $\cB$ is the quaternion algebra over $\Q$ ramified at the
primes $2$ and $3$. The field $\Q(\sqrt{6})$ is isomorphic to a
maximal subfield of $\cB$, and by Proposition \ref{prop: C2} there
exists a morphism $\psi\colon G_\Q\ra \cB^\times /\Q^\times $ such
that $\overline{\delta(\psi)}=(-3,6)_P$ and
$\delta(\psi)_\pm=(-3,1)_\Q$ which is trivial in $H^2(G_\Q,\{\pm 1
\})$. Therefore $\gamma=\delta(\psi)$ and we deduce the existence
of an abelian variety defined over $\Q$ and isogenous to $B$.

\subsection*{Example.} Let $B$ be the only quaternionic building
block of dimension 2 associated to a modular form $f$ of level
$N=60$ with Nebentypus of order 4. In this case the variety $A_f$
is 4-dimensional and the cohomology class associated to $B$ has
components $\overline{\gamma}=(5,2)_P\cdot(-3,5)_P$, and
$\gamma_\pm$  the quaternion algebra over $\Q$ ramified at the
primes $3$ and $5$.
 The field $K_P=\Q(\sqrt{5},\sqrt{-3})$ is the minimum field of
definition of the variety and of its endomorphisms up to isogeny,
and the algebra $\cB=\End_\Qb^0(B)$ is the quaternion algebra over
$\Q$ ramified at $2$ and $5$, which is isomorphic to $(-2,5)_\Q$.
Hence, by Proposition \ref{prop: C2xC2} there exists a $\psi\colon
G_\Q\ra \cB^\times /\Q^\times $ such that
$\overline{\delta(\psi)}=(5,-2)_P\cdot (-3,5)_P$ and
$\delta(\psi)_\pm=(5,3)_\Q$, which is the quaternion algebra
ramified at $3$ and $5$. Hence $\gamma=\delta(\psi)$ and by
Proposition  \ref{prop: descens del cos de definicio mantenint el
cos de definicio dels endomorfismes} there exists a variety $B_0$
defined over $\Q$ and with all its endomorphisms defined over
$K_P$ that is isogenous to $B$.

\subsection*{Example.} Let $B$ be the only quaternionic  building
block of dimension $2$  associated to a newform $f$ of level
$N=80$ and Nebentypus of order $4$. Now
$\overline\gamma=(5,2)_P\cdot(-4,3)_P$ and $\gamma_\pm$ is the
quaternion algebra over $\Q$ ramified at $2$ and $5$. Again $K_P$,
which in this case is $\Q(\sqrt{5},\sqrt{-1})$, is the minimum
field of definition of $B$ and of its endomorphisms up to isogeny.

First, we observe that there does not exist a variety $B_0$
defined over $\Q$ and with all its endomorphisms defined over
$K_P$. By~\ref{prop: descens del cos de definicio mantenint el cos
de definicio dels endomorfismes} the existence of such variety
would be equivalent to the existence of a $\psi\colon  G_\Q\ra
\cB^\times /\Q^\times $ with image isomorphic to $C_2\times C_2$
such that $\overline{\delta(\psi)}=\overline \gamma$ and
$\delta(\psi)_\pm = \gamma_\pm$. By~\ref{prop: C2xC2},
$\overline{\delta(\psi)}=(s,a)_P\cdot (t,b)_P$ with $\cB\simeq
(a,b)_\Q$. If we want $\overline{\delta(\psi)}=\overline{\gamma}$,
the only possibilities for $a,b$ modulo squares are the following:
$a=2$ and $b=3$, $a=2$ and $b=-3$, $a=-2$ and $ b=3$ or $a=-2$ and
$ b=-3$. Since $\cB$ is the quaternion algebra of discriminant
$6$, only the first two options are possible. But if
$\overline{\delta(\psi)}=(5,2)_P\cdot(-4,3)_P$, from~\ref{prop:
C2xC2} we see that $\delta(\psi)_\pm=(5,-4)_\Q$, which is not
equal to $\gamma_\pm$, and if
$\overline{\delta(\psi)}=(5,2)_P\cdot(-4,-3)_P$ then
$\delta(\psi)_\pm=(-5,-4)_\Q$ which is also not equal to
$\gamma_\pm$. Hence there does not exist such a $\psi$.

Now we will see that there exists a  $\psi\colon G_\Q\ra
\cB^\times /F^\times $ with image isomorphic to  $D_{2\cdot 4}$
such that $\gamma = \delta(\psi)$. This will tell us that there
exists an abelian variety $B_0$ defined over $\Q$ that is
isogenous to $B$, but that does not have all its endomorphisms
defined over $K_P$. First of all, we observe that  $\cB\simeq
(-1,3)_\Q$, and so  $\cB$ contains a maximal subfield isomorphic
to $\Q(i)$, where $i=\sqrt{-1}$. This implies that $\cB^\times
/\Q^\times $ contains subgroups isomorphic to  $D_{2\cdot 4}$.
More precisely, if  $x,y$ are elements in $\cB$ such that
$x^2=-1$, $ y^2=3$, and $xy=-yx$, then the subgroup of $\cB^\times
/\Q^\times $ generated by $\overline{1+x}$ and $\overline y$ is
isomorphic to $D_{2\cdot 4}$.

The number field $M=\Q( {\sqrt[4]{5}},i)$  has $\Gal(M/\Q)\simeq
D_{2\cdot 4}$, generated by the automorphisms $\sigma\colon
{\sqrt[4]{5}}\mapsto i{\sqrt[4]{5}},\ i\mapsto i$ and $\tau\colon
{\sqrt[4]{5}}\mapsto {\sqrt[4]{5}},\ i\mapsto -i$. We define
$\psi\colon G_\Q\ra \cB^\times /F^\times $ as the morphism sending
$\sigma$ to $\overline{1+x}$ and $\tau$ to $\overline y$. From the
expressions given in Proposition \ref{prop: diedral} we see that
$\overline{\delta(\psi)}=(-1,3)_P\cdot(5,2)_P$, which is equal to
$\overline{\gamma}$. It only remains to see that
$\delta(\psi)_\pm=\gamma_\pm$.  Let $D$ be the quaternion algebra
associated to $\delta(\psi)_\pm$. Since $\delta(\psi)_\pm\in
Z^2(\Gal(M/\Q),\{\pm 1\})$ and the extension  $M/\Q$ only ramifies
at the primes $2$ and  $5$, $D$ can only ramify at the places $2$,
$5$ and $\infty$ (see~\cite[Proposition 18.5]{Pi}). We will see
that $D\otimes_\Q\Q(i)$ is not trivial in the Brauer group (and
therefore $D$ ramifies at some prime), and that $D\otimes_\Q
\Q(\sqrt{5})$ is trivial (and therefore $D$ does not ramify at
$\infty$). These two conditions  imply that $D$ ramifies exactly
at $2$ and $5$.

Since $\Gal(M/\Q(i))=\langle\sigma\rangle$, a 2-cocycle $c$
representing $D\otimes_\Q \Q(i)$ is the restriction  to the
subgroup $\langle\sigma\rangle\subseteq \Gal(M/\Q)$ of a cocycle
representing $\delta(\psi)_\pm$. From~\eqref{eq:delta psi +- en el
cas D2n}
 we obtain that
\begin{eqnarray*}
c(\sigma^i,\sigma^{j})=
\begin{cases}
1 & \text{if $i+j< 4$}
\\
 -1 & \text{if $i+j\geq 4$.}
\end{cases}
\end{eqnarray*}
By \cite[Lemma 15.1]{Pi} the algebra associated to this cocycle is
trivial if and only if  $-1\in \mathrm{Nm}_{M/\Q(i)}(M)$, where
$\mathrm{Nm}_{M/\Q(i)}$ refers to the norm in the extension
$M/\Q(i)$. But $-1$  is not a norm of this extension, hence
$D\otimes_\Q \Q(i)$ is non-trivial in the Brauer group.

Since  $\Gal(M/\Q(\sqrt{5}))=\langle\sigma^2,\tau\rangle$, a
2-cocycle $c$ representing $D\otimes_\Q \Q(\sqrt{5})$ is the
restriction to $\langle\sigma^2,\tau\rangle\subseteq \Gal(M/\Q)$
of a cocycle representing $\delta(\psi)_\pm$. Again
from~\eqref{eq:delta psi +- en el cas D2n} we obtain the
following:
\[
\begin{array}{llll}
c(1,1)=1 \ \ & c(\sigma^2,1)=1\ \ &
c(\tau,1)=1 \ \ & c(\sigma^2\tau,1)=1  \\
c(1,\sigma^2)=1 \ \ & c(\sigma^2,\sigma^2)=-1 \ \ &
c(\tau,\sigma^2)=-1 \ \ & c(\sigma^2\tau,\sigma^2)=1  \\
c(1,\tau)=1 \ \ & c(\sigma^2,\tau)=1 \ \ &
c(\tau,\tau)=1 \ \ & c(\sigma^2\tau, \tau)=1  \\
c(1,\sigma^2\tau)=1 \ \ & c(\sigma^2,\sigma^2\tau)=-1 \ \ &
c(\tau,\sigma^2\tau)=-1 \ \ & c(\sigma^2\tau,\sigma^2\tau)=1.
\end{array}
\]
To see that the cohomology class of this cocycle in
$H^2(\Gal(M/\Q(\sqrt{5})),M^\times )$ is trivial (where now the
action is the natural Galois action), we define a map $\lambda$ by
$\lambda(1)=1$, $\lambda(\sigma^2)=i$, $\lambda(\tau)=i$ and
$\lambda(\sigma^2\tau)=-i$. Now a computation shows that
$c(\rho,\mu)=\lambda(\rho)\cdot{^\rho\lambda(\mu)}\cdot\lambda(\rho\mu)^{-1}$,
for all $\rho,\mu\in\Gal(M/\Q(\sqrt{5}))$.

\subsection*{Example.} Consider the building block $B$ in the table
associated with a newform of  conductor 336. For this variety
$\overline{\gamma}=(-3,11)_P$ and $\gamma_\pm$ is the quaternion
algebra ramified at $2$ and $3$. Hence $K_P=\Q(\sqrt{-3})$ and
since  $\Res_\Q^{K_P}(\gamma_\pm)= 1$ we have that $K_P$ is the
minimum field of definition of $B$ and of its endomorphisms up to
isogeny. We will show that $B$ is not isogenous to any variety
defined over $\Q$.

As $K_P$ is a quadratic number field and $\gamma_\pm\neq 1$, the
only morphisms $\psi$ we have to consider are those with image
isomorphic to $C_2$ or to $C_n$ for some even $n>2$. The only such
values of $n$ with $\cB^\times /\Q^\times $ containing a subgroup
isomorphic to $C_n$ are $n=4$ and $n=6$. Since the component
$\overline{\delta(\psi)}$ associated to a $\psi$ with image $C_n$
has the form $(t,2+\zeta_n+\zeta_n^{-1})$, and for $n=4,6$ we have
that $2+\zeta_n+\zeta_n^{-1}$ is not congruent to $ 11$ modulo
$\{\pm 1\}\Q^{*2}$, it turns out that there does not exist any
$\psi$ with image $C_4$ or $C_6$ such that $\gamma= \delta(\psi)$.
If $\psi$ has image $C_2$, the only possibilities are
$\overline{\delta(\psi)}=(-3,11)$ or
$\overline{\delta(\psi)}=(-3,-11)$. In the first case we would
have $\delta(\psi)_\pm=(-3,1)_\Q$ and in the second case
$\delta(\psi)_\pm=(-3,-1)$. In both cases $\delta(\psi)_\pm\neq
\gamma_\pm$, and thus there does not exist a $\psi$ with image
$C_2$ such that $\gamma=\delta(\psi)$.

\subsection*{Acknowledgements}
I am grateful to Jordi Quer for his guidance and help throughout this work. This research was partially supported by
Grants MTM2009-13060-C02-01 and 2009 SGR 1220.

\end{document}